\documentclass[12pt]{amsart}
\usepackage{amsthm} %%% Defines \theoremstyle
\usepackage{amsmath} %%% Defines \numberwithin, \operatorname
\usepackage{amssymb} %%% Defines \mathbb

\usepackage{microtype} 
\usepackage{pinlabel} 

\usepackage[
hidelinks,
pagebackref,
pdftex]{hyperref}

\renewcommand*{\backref}[1]{}
\renewcommand*{\backrefalt}[4]{
  \ifcase #1 
  [No citations.]
  \or [#2]
  \else [#2]
  \fi }

\let\originalleft\left
\let\originalright\right
\renewcommand{\left}{\mathopen{}\mathclose\bgroup\originalleft}
\renewcommand{\right}{\aftergroup\egroup\originalright}

\newcommand{\calC}{\mathcal{C}}

\newcommand{\calK}{\mathcal{K}}

\newcommand{\calX}{\mathcal{X}}

\newcommand{\sE}{{\sf E}}

\newcommand{\sN}{{\sf N}}
\newcommand{\sS}{{\sf S}}
\newcommand{\sW}{{\sf W}}

\newcommand{\ZZ}{\mathbb{Z}}

\newcommand{\st}{\mathbin{\mid}} %%% \mathbin = binary operator,
\newcommand{\from}{\colon} % As in ``f maps _from_ X _to_ Y''.

\newcommand{\closure}[1]{{\overline{#1}}}
\newcommand{\frontier}{\operatorname{fr}} 
\newcommand{\bdy}{\partial} % Boundary
\newcommand{\Teich}{{Teichm\"uller~}} 

\input{header_article.tex}
\newcommand{\fakeenv}{} %%% prints the emptystring

\newenvironment{restate}[2]  %%% restate takes two arguments 
{ 
 \renewcommand{\fakeenv}{#2} %%% So now \fakeenv prints #2
 \theoremstyle{plain} 
 \newtheorem*{\fakeenv}{#1~\ref{#2}} %%% so now #2 is the name of a
 \begin{\fakeenv}
}
{
 \end{\fakeenv}
}

\newcommand{\AC}{\mathcal{AC}}
\newcommand{\Uniform}{{\sf U}}
\newcommand{\outdist}{\operatorname{outer}}
\newcommand{\indist}{\operatorname{inner}}

\DeclareMathOperator{\I}{i}

\begin{document}

\title[Uniform hyperbolicity]{Uniform hyperbolicity of the curve graph
  via surgery sequences}

\author[Clay]{Matt Clay}
\author[Rafi]{Kasra Rafi}
\author[Schleimer]{Saul Schleimer}

\thanks{\tiny \noindent The first author is partially supported by NSF
  grant DMS-1006898.  The second author is partially supported by NSF
  grant DMS-1007811.  The third author is partially supported by
  EPSRC grant EP/I028870/1.\\ This work is in the public domain.}

\date{\today}

\begin{abstract}
We prove that the curve graph $\calC^{(1)}(S)$ is Gromov-hyperbolic
with a constant of hyperbolicity independent of the surface $S$.  The
proof is based on the proof of hyperbolicity of the free splitting
complex by Handel and Mosher, as interpreted by Hilion and Horbez.
\end{abstract}

%%% MSC numbers: 57M99 (low-dimensional topology), 32F60 (Teichmuller theory)
%%% keywords - curve complex, arc complex, Gromov hyperbolic

\maketitle

\section{Introduction}
\label{Sec:Intro}

In recent years the curve graph has emerged as the central object in a
variety of areas, such as Kleinian groups~\cite{Minsky99, Minsky10,
  BrockEtAl12}, \Teich spaces~\cite{Rafi05, Rafi10, BrockEtAl10} and
mapping class groups~\cite{MasurMinsky00, BehrstockEtAl12}. The
initial breakthrough was the result of Masur and Minsky showing that
the curve graph is Gromov hyperbolic~\cite{MasurMinsky99}.

In this note, we give an new proof of the hyperbolicity of all curve
graphs.  We improve on the original proof by additionally showing that
the hyperbolicity constants are \emph{uniform}: that is, independent
of the topology of the surface.

We use the same hyperbolicity criterion as defined and used by Masur
and Minsky~\cite[Definition~2.2]{MasurMinsky99}.  Suppose $\calX$ is a
graph, equipped with a family of paths, and each path $\sigma$ is
equipped with a projection map $\pi_\sigma \from \calX \to \sigma$.
If the family of paths and projection maps satisfy the
\emph{retraction}, \emph{Lipschitz}, and \emph{contraction} axioms, as
stated in \refsec{AC} then $\calX$ is
hyperbolic~\cite[Theorem~2.3]{MasurMinsky99}.  We also provide a proof
in \refsec{Hyperbolicity}.  Bestvina and Feighn recently used a
similar argument to show that the \emph{free factor graph} of a free
group is Gromov hyperbolic~\cite{BestvinaFeighn11}.

For the curve graph and for the free factor graph another, more
geometric, space played the key role in the definition of paths and
projection maps.  For the curve graph this was \emph{\Teich space};
for the free factor graph it was \emph{outer space}.  An understanding
of geodesics in the geometric spaces was necessary to define the
family of paths and their projection maps.

The \emph{splitting graph}, another variant of the curve graph for the
free group, was recently shown to be hyperbolic by Handel and
Mosher~\cite{HandelMosher11}.  They also use the hyperbolicity
criterion of Masur and Minsky.  A novel aspect of their approach was
to dispense with the ancillary geometric space; instead they define
projection as if the space \emph{were} hyperbolic, and the family of
paths were geodesics.  Specifically, given three points $x$, $y$ and
$z$ in the space, the projection of $z$ to the path $\sigma$ from $x$
to $y$ is the first point along $\sigma$ that is close (in a uniform
sense) to the path from $z$ to $y$.  See \reffig{Projection}.

\begin{figure}[htbp]
\centering
\labellist
\small\hair 2pt
\pinlabel {$x$} [Br] at 5 26
\pinlabel {$y$} [Bl] at 240 26
\pinlabel {$z$} [Bl] at 123 201
\pinlabel {$\pi(z)$} [t] at 152 22
\endlabellist
\includegraphics[height = 3.5 cm]{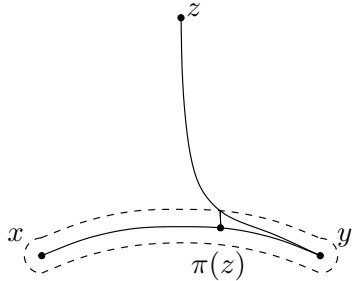}
\caption{Handel--Mosher projection of a point $z$ to the path from
  $x$ to $y$.}
\label{Fig:Projection}
\end{figure}

The paths used by Handel and Mosher in the splitting graph have a
key property that is very reminiscent of negatively curved spaces:
\emph{exponential divergence}.  In the other direction we find
exponential convergence.  On a small scale, Handel and Mosher show
paths that start distance two apart, and that have the same target,
must ``intersect'' after a distance depending only on the rank of the
free group.  On a larger scale, this implies that the ``girth'' of two
paths, with the same target, is cut in half after a similar distance.
This property is the main tool used to verify the Masur and Minsky
axioms.

Hilion and Horbez~\cite{HilionHorbez12} gave a geometric spin to
Handel and Mosher's argument; this led them to an alternative proof of
hyperbolicity of the splitting graph (in their setting called the
\emph{sphere graph}).  Their paths were surgery sequences of spheres
in the doubled handlebody.  We closely follow their set-up and use
surgery sequences of arcs and curves as paths in the curve graph.  We
now state our main results.

Let $S = S_{g,n}$ be a surface of genus $g$ with $n$ boundary
components, let $\calC(S)$ be the complex of curves, and let $\AC(S)$
be the complex of curves and arcs; we defer the definitions to
\refsec{Back}.  We add a superscript $(1)$ to denote the one-skeleton.

\begin{restate}{Theorem}{Thm:AC}
There is a constant $\Uniform$ such that if $3g - 3 + n \geq 2$ and $n
> 0$ then $\AC^{(1)}(S_{g,n})$ is $\Uniform$--hyperbolic.
\end{restate}

\noindent The inclusion $\calC^{(1)}(S_{g,n}) \to \AC^{(1)}(S_{g,n})$
gives a quasi-isometric embedding with constants independent of $g$
and $n$.  Deduce the following.

\begin{restate}{Corollary}{Cor:Curve}
There is a constant $\Uniform$ such that if $3g - 3 + n \geq 2$ and $n
> 0$ then $\calC^{(1)}(S_{g,n})$ is $\Uniform$--hyperbolic.
\end{restate}

\noindent We also prove uniform hyperbolicity in the closed case, when
$n = 0$.  This follows from \refthm{AC}, as $\calC^{(1)}(S_{g,0})$
isometrically embeds in $\calC^{(1)}(S_{g,1})$.

\begin{restate}{Theorem}{Thm:Closed}
There is a constant $\Uniform$ such that if $3g - 3 \geq 2$ then
$\calC^{(1)}(S_{g})$ is $\Uniform$--hyperbolic.
\end{restate}

As noted above, the various constants appearing in our argument are
uniform.  This is mostly due to \reflem{Common} which shows that paths
that start distance two apart, and that have the same target, must
``intersect'' after a uniform distance.

After the original paper of Masur and Minsky,
Bowditch~\cite{Bowditch06} and Hamenst\"adt~\cite{Hamenstadt07} also
gave proofs of the hyperbolicity of the curve graph.  In all of these
the upper bound on the hyperbolicity constant depended on the topology
of the surface $S$.  During the process of writing this paper, several
other proofs of uniform hyperbolicity emerged.
Bowditch~\cite{Bowditch12} has refined his approach to obtain uniform
constants using techniques he developed in~\cite{Bowditch06}; the
proof by Aougab~\cite{Aougab12} has many common themes with the work
of Bowditch.  The work of Hensel, Przytycki, and
Webb~\cite{HenselEtAl13} also uses surgery paths and has other points
of contact with our work.  However Hensel, Przytycki, and Webb do not
use the Masur--Minsky criterion; they also obtain much smaller
hyperbolicity constants than given here.

\subsection*{Acknowledgements} 

We thank the Centre de Recerca Matem\`atica for its hospitality during
its 2012 research program on automorphisms of free groups.

% Local Variables:
% mode: latex
% TeX-master: "arcs"
% End:

\section{Background}
\label{Sec:Back}

Let $S = S_{g,n}$ be a connected, compact, oriented surface of genus
$g$ with $n$ boundary components.  We make the standing assumption
that the \emph{complexity} of $S$, namely $3g - 3 + n$, is at least
two.  This rules out three surfaces: $S_{0,4}, S_{1}, S_{1,1}$.  In
each case the arc and curve complex is a version of the Farey graph;
the Farey graph has hyperbolicity constant one when we restrict to
the vertices, and $3/2$ when we include the edges.

%%% It is at least 3/2 as that is the diameter of one-skeleton of a
%%% triangle.  Now suppose p, q, r are vertices of \calF, the Farey
%%% graph, connected by geodesics R = [p,q], P = [q,r], and Q = [r,p].
%%% Note that any every of \calF separates.  Thus any edges crossed by
%%% R is also crossed by the path \bar{Q} \cup \bar{P}. So the
%%% vertices of R are in a one-neighborhood of \bar{Q} \cup \bar{P}.
%%% The extra 1/2 takes care of the midpoints of edges.  If we include
%%% the faces then the constant is \sqrt(3)/2.  I think. 

\subsection{Arcs and curves}

A properly embedded curve or arc $\alpha \subset S$ is
\emph{essential} if $\alpha$ does not cut a disk off of $S$.  A
properly embedded curve $\alpha$ is \emph{non-peripheral} if it does
not cut an annulus off of $S$.  Define $\AC(S)$ to be the set of
ambient isotopy classes of essential arcs and essential non-peripheral
curves.

For classes $\alpha, \beta \in \AC(S)$ define the geometric
intersection number $\I(\alpha, \beta)$ to be the minimal intersection
number among representatives.  A non-empty subset $A \subset \AC(S)$
is a \emph{system of arcs and curves}, or simply a \emph{system}, if
for all $\alpha, \beta \in A$ we have $\I(\alpha, \beta) = 0$.  We now
give $\AC(S)$ the structure of a simplicial complex by taking systems
for the simplices.  We use $\calC(S)$ to denote the subcomplex of
$\AC(S)$ spanned by curves alone.  Note that these are flag complexes:
when the one-skeleton of a simplex is present, so is the simplex
itself.  Let $\calK^{(1)}$ denote the one-skeleton of a simplicial
complex $\calK$.

If $\alpha$ and $\beta$ are vertices of $\AC(S)$ then we use
$d_S(\alpha, \beta)$ to denote the combinatorial distance coming from
$\AC^{(1)}(S)$.  Given two systems $A, B \subset \AC(S)$ we define their
\emph{outer distance} to be
\[
\outdist(A, B) = \max \{ d_S(\alpha, \beta) \st \alpha \in A, \,\beta \in
B \}
\]
and their \emph{inner distance} to be
\[
\indist(A, B) = \min \{ d_S(\alpha, \beta) \st \alpha \in A, \,\beta \in
B \}.
\]
%%% So \indist \leq \outdist \leq \indist + 2.  
For $\beta \in \AC(S)$ we write $\indist(A, \beta)$ instead of
$\indist(A, \{\beta\})$, and similarly for the outer distance.  If $A$
and $B$ are systems and $C \subset B$ is a subsystem then
\begin{equation}
\label{Eqn:Subsystem}
\indist(A,B) \leq \indist(A,C) \leq \indist(A, B) + 1.
\end{equation}
For any three systems $A$, $B$, and $C$ there is a triangle
inequality, up to an additive error of one, namely
\begin{equation}
\label{Eqn:Triangle}
\indist(A, B) \leq \indist(A, C) + \indist(C, B) + 1.
\end{equation}
%%% Because \indist is a minimum, and because diam(C) <\leq 1.
The additive error can be reduced to zero when $C$ is a singleton.
%%% There is also a triangle inequality for \oudist, namely
%%% \outdist(A,B) \leq \outdist(A,C) + \outdist(C,B) + 3.

Suppose $A \subset \AC(S)$ is a system and $\gamma \in \AC(S)$ is an
arc or curve.  We say $\gamma$ \emph{cuts} $A$ if there is an element
$\alpha \in A$ so that $\I(\gamma, \alpha) > 0$.  If $\gamma$ does not
cut $A$ then we say $\gamma$ \emph{misses} $A$.

A system $A$ \emph{fills} $S$ if every curve $\gamma \in \calC(S)$
cuts $A$.  Note that filling systems are necessarily comprised solely
of arcs.  A filling system $A$ is \emph{minimal} if no subsystem is
filling.  

\begin{lemma}
\label{Lem:Count}
Suppose $S = S_{g,n}$, with $n > 0$, and suppose $A$ is a minimal
filling system.  If $S - A$ is a disk then $|A| = 2g - 1 + n$.  On the
other hand, if $S - A$ is a collection of peripheral annuli then $|A|
= 2g - 2 + n$. \qed
\end{lemma}

\subsection{Surgery}

If $X$ is a space and $Y \subset X$ is a subspace, let $N = N_X(Y)$
denote a small regular neighborhood of $Y$ taken in $X$.  Let
$\frontier(N) = \closure{\bdy N - \bdy X}$ be the \emph{frontier} of
$N$ in $X$.

Now suppose $A$ is a system and $\omega$ is a directed arc cutting
$A$.  Choose representatives to minimize intersection numbers between
elements of $A$ and $\omega$.  Suppose $\delta$ is the component of
$\omega - A$ containing the initial point of $\omega$.  Thus $\delta$
meets only one component of $A$, say $\alpha$; we call $\alpha$ the
\emph{active element} of $A$.  Let $N = N_S(\alpha \cup \delta)$ be a
neighborhood.  Let $N'$ be the component of $N - \alpha$ containing
the interior of $\delta$.  Let $\alpha^\omega$ be the component(s) of
$\frontier(N)$ that are contained in $N'$.  See \reffig{Surgery} for
the two possible cases.

\begin{figure}[htbp]
\labellist
\small\hair 2pt
 \pinlabel {$\delta$} [Bl] at 127 102
 \pinlabel {$\alpha^\omega$} [tl] at 160 99
 \pinlabel {$\alpha$} [Bl] at 49 151
 \pinlabel {$\bdy S$} [bl] at 199 3
\endlabellist
\[
\begin{array}{ccc}
\includegraphics[height = 3.5 cm]{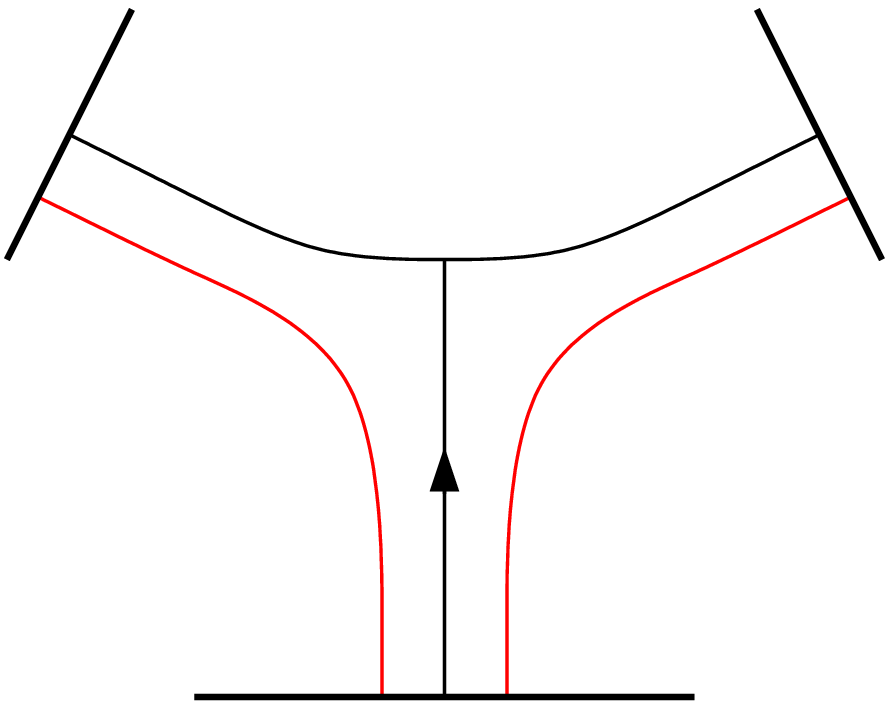} && 
\includegraphics[height = 3.5 cm]{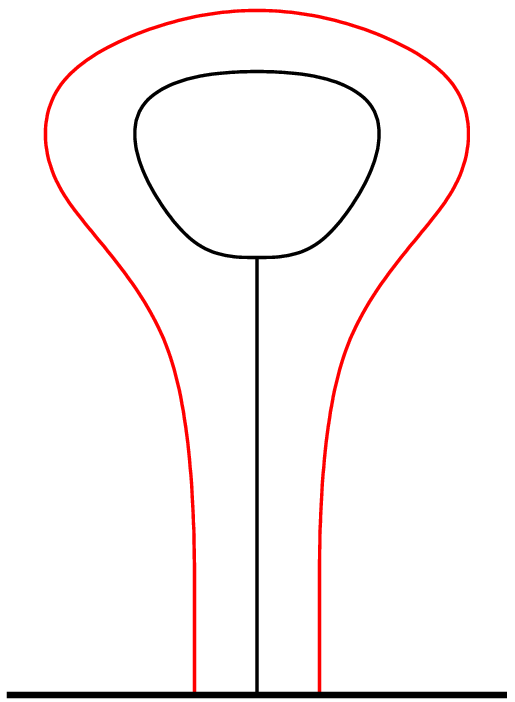}
\end{array}
\]
\caption{The result of surgery, $\alpha^\omega$, is either a pair of arcs or
  a single arc as $\alpha$ is an arc or a curve.}
\label{Fig:Surgery}
\end{figure}

We call the arcs of $\alpha^\omega$ the \emph{children} of $\alpha$.
Define $A^\omega = (A - \alpha) \cup \alpha^\omega$; this is the
result of \emph{surgering} $A$ exactly once along $\omega$.

\begin{lemma}
\label{Lem:Continuous}
Suppose $A, B$ are systems and $\omega$ is a directed arc cutting $A$.
Then $|\indist(A^\omega, B) - \indist(A, B)| \leq 1$. 
\end{lemma}

\begin{proof}
Note that $A^\omega \cup A$ is again a system.  The conclusion now
follows from two applications of \refeqn{Subsystem}.
\end{proof}

When $B = \{\omega\}$ a stronger result holds. 

\begin{proposition}
\label{Prop:Monotonic}
Suppose $A$ is a system and $\omega$ is a directed arc cutting $A$.
Then $\indist(A^\omega, \omega) \leq \indist(A, \omega)$. 
\end{proposition}

\begin{proof}
We induct on $\indist(A, \omega)$.  Suppose that $\indist(A, \omega) =
n + 1$.  Let $\alpha$ be the element of $A$ realizing the minimal
distance to $\omega$.  There are two cases.  If $\alpha$ is not the
active element then $\alpha \in A^\omega$ and the inner distance
remains the same or decreases.  For example, this occurs when $n = 0$.

Suppose, instead, that $\alpha$ is the active element and that $n >
0$.  Pick $\beta \in \AC(S)$ with
\begin{itemize}
\item
$d_S(\alpha, \beta) = 1$, 
\item
$d_S(\beta, \omega) = n$, and,
\item
subject to the above, $\beta$ minimizes $\I(\beta,\omega)$.
\end{itemize}
Consider the system $B = \{\alpha, \beta\}$.  The induction hypothesis
gives $\indist(B^\omega, \omega) \leq \indist(B, \omega)$.  If $\beta$
is the active element of $B$ then we contradict the minimality of
$\beta$.  Thus $\alpha$ is the active element of $B$.  We deduce
$\indist(\alpha^\omega, \omega) \leq d_S(\alpha, \omega)$, competing
the proof.
\end{proof}

If $A$ is a system and $\omega$ is a directed arc cutting $A$ then we
define a \emph{surgery sequence} starting at $A$ with \emph{target}
the directed arc $\omega$, as follows.  Set $A_0 = A$ and let $A_{i+1}
= A_i^\omega$; that is, we obtain $A_{i+1}$ by surgering the active
element of $A_i$ exactly once along $\omega$.  The arc $\omega$ misses
the last system $A_N$; the resulting sequence is $\{A_i\}_{i=0}^N$.

Given integers $i \leq j$ we adopt the notation $[i,j] = \{ k \in \ZZ
\st i \leq k \leq j \}$.

\begin{lemma}
\label{Lem:Inner}
Suppose $\{ A_i \}_{i = 0}^N$ is a surgery sequence with target
$\omega$.  Then for each distance $d \in [0, \indist(A, \omega) - 1]$
there is an index $i \in [0,N]$ such that $\indist(A,A_i) = d$.
\end{lemma}

\begin{proof}
Since $\outdist(A_N, \omega) \leq 1$ the triangle inequality
\[
\indist(A, \omega) \leq \indist(A, A_N) + \indist(A_N, \omega)
\]
holds without additive error.  Thus $\indist(A, A_N) \geq \indist(A,
\omega) - 1$.  The conclusion now follows from \reflem{Continuous}.
\end{proof}

We can also generalize \refprop{Monotonic} to sequences.  As we do not
use this in the remainder of the paper, we omit the proof. 

\begin{proposition}
\label{Prop:MonotonicTwo}
Suppose $\{ A_i \}_{i = 0}^N$ is a surgery sequence with target
$\omega$.  Let $\alpha_k \subset A_k$ be the active element and set
$\omega_k = \alpha_k^\omega$.  Then $\indist(A_{i+1}, \omega_k) \leq
\indist(A_i, \omega_k)$, for $i < k$.  \qed
\end{proposition}

%%% Proof sketch - \omega_k and \omega are identical along the initial
%%% segment of each.  Thus surgery of A_i, for i < k, with target
%%% \omega or target \omega_k, is identical.  The proposition now
%%% follows from \refprop{Monotonic}.

Suppose $B \subset A$ is a subsystem and $\omega$ is a directed arc
cutting $A$.  Let $\{ A_i \}$ be the surgery sequence starting at $A$
with target $\omega$.  Let $B_0 = B$ and suppose we have defined $B_i
\subset A_i$.  If the active element $\alpha \in A_i$ is \emph{not} in
$B_i$ then we define $B_{i+1} = B_i$.  If the active element $\alpha
\in A_i$ is in $B_i$ then define $B_{i+1} = B_i^\omega$.  In any case
we say that the elements of $B_{i+1}$ are the \emph{children} of the
elements of $B_i$; for $j \geq i$ we say that the elements of $B_j$
are the \emph{descendants} of $B_i$.  We call the sequence $\{B_i\}$ a
surgery sequence with \emph{waiting times}; the sequence $\{B_i\}$ is
\emph{subordinate} to $\{A_i\}$.

\section{Descendants}
\label{Sec:Descend}

The goal of this section is to prove \reflem{Common}: disjoint systems
have a common descendant within constant distance.  Recall that a
simplex $A \subset \AC(S)$ is called a system.

\begin{lemma}
\label{Lem:Cut}
Suppose $A$ is a system and $\omega$ is a directed arc cutting $A$.
Suppose $\gamma \in \calC(S)$ is a curve.  If $\gamma$ cuts $A$ then
$\gamma$ cuts $A^\omega$.
\end{lemma}

\begin{proof}
Suppose $\alpha \in A$ is the active element.  If $\gamma$ cuts some
element of $A - \alpha$ then there is nothing to prove.  If $\gamma$
cuts $\alpha$ then, consulting \reffig{Surgery}, the curve $\gamma$
also cuts $\alpha^\omega$ and so cuts $A^\omega$.
\end{proof}

\begin{lemma}
\label{Lem:Fill}
Suppose $\{A_i\}$ is a surgery sequence with target $\omega$.  For any
index $k$, if $\outdist(A_0, A_k) \geq 3$ then $A_j$ is filling for
all $j \geq k$.
\end{lemma}

\begin{proof}
By \reflem{Cut} it suffices to prove that $A_k$ is filling.  Pick any
$\gamma \in \calC(S)$.  Since $\outdist(A_0, A_k) \geq 3$ it follows
that $\gamma$ cuts $A_0$ or $A_k$, or both.  If $\gamma$ cuts $A_k$ we
are done.  If $\gamma$ cuts $A_0$ then we are done by \reflem{Cut}.
\end{proof}

\begin{lemma}
\label{Lem:Common}
Suppose $A$ is a system and $\omega$ is a directed arc with
$\indist(A, \omega) \geq 6$.  Suppose $B, C \subset A$ are subsystems.
Let $\{ A_i \}_{i = 0}^N$ be the surgery sequence starting at $A_0 =
A$ with target $\omega$.  Let $\{ B_i \}$ and $\{ C_i \}$ be the
subordinate surgery sequences.  Then there is an index $k \in [0, N]$
such that:
\begin{enumerate}
\item $B_k \cap C_k \neq \emptyset$ and 
\item $\indist(A_0, A_i) \leq 5$ for all $i \in [0, k]$.
\end{enumerate}
\end{lemma}

We paraphrase this as ``the subsystems $B$ and $C$ have a common
descendant within constant distance of $A$''.

\begin{proof}[Proof of \reflem{Common}]
Let $\ell$ be the first index with $\indist(A, A_\ell) = 3$.  Note
that $\ell$ exists by \reflem{Inner}.  Also, \reflem{Continuous}
implies that $\indist(A, A_{\ell-1}) = 2$.  Suppose $\beta$ is the
active element of $A_{\ell-1}$.  It follows that $\indist(A, \beta) =
2$ and $\beta$ is the only element of $A_{\ell-1}$ with this inner
distance to $A$.  Thus every $\alpha \in A_\ell$ has inner distance
three to $A$.  If $\omega$ misses some element of $A_\ell$ then
$\indist(A, \omega) \leq 4$, contrary to hypothesis.  Thus $\omega$
cuts every element of $A_\ell$.  Isotope the arcs of $A_\ell$ to be
pairwise disjoint and to intersect $\omega$ minimally.

If $B_\ell \cap C_\ell \neq \emptyset$ then we take $k = \ell$ and we
are done.  Suppose instead $B_\ell$ and $C_\ell$ are disjoint.  Since
$\indist(A, A_\ell) = 3$ we have both $\outdist(B, B_\ell)$ and
$\outdist(C, C_\ell)$ are at least three.  Deduce from \reflem{Fill}
that $B_\ell$ and $C_\ell$ both fill $S$, and thus consist only of
arcs.  Let $B' \subset B_\ell$ and $C' \subset C_\ell$ be minimal
filling subsystems. 

Set $x = -\chi(S) = 2g - 2 + n$.  Set $b = 1$ if $S - B'$ is a disk.
Set $b = 0$ if $S - B'$ is a union of peripheral annuli.
\reflem{Count} implies $|B'| = x + b$.  Define $c$ similarly, with
respect to $C'$.  Let $A' = B' \cup C'$.  Let $p$ be the number of
peripheral annuli in $S - A'$.  Observe that if either $b$ or $c$ is
one, then $p$ is zero.

We build a graph $G$, \emph{dual} to $A'$, as follows.  For every
component $C \subset S - A'$ there is a dual vertex $v_C$.  For every
arc $\alpha \in A'$ there is a dual edge $e_\alpha$; the two ends
$e_\alpha$ are attached to $v_C$ and $v_D$ where $C$ and $D$ meet the
two sides of $\alpha$.  Note the possibility that $C$ equals $D$.
Finally, for every peripheral annulus component $P \subset S - A'$
there is a peripheral edge $e_P$.  Both ends of $e_P$ are attached to
$v_P$.

Thus $G$ has $|A'| + p = 2x + b + c + p$ edges.  Since $S$ is homotopy
equivalent to $G$, we deduce that $G$ has $x + b + c + p$ vertices.
%%% V - E = V - 2x - p - b - c = -x
%%% Note: if b or c = 1 then p = 0. 
Since $B' \cap C' = \emptyset$, the graph $G$ has no vertices of
degree one or two.  

\begin{claim*}
One of the following holds.
\begin{enumerate}
\item 
The graph $G$ has a vertex of valence three, dual to a disk component
of $S - A'$.
\item 
Every vertex of $G$ has valence four and every component of $S - A'$
is a disk.
\end{enumerate}
\end{claim*}

\begin{proof}[Proof of Claim]
Let $V_d$ denote the number of vertices of $G$ with degree $d$.  As
there are no vertices of valence one or two, twice the number of edges
of $G$ equals $\sum_{d \geq 3} d \cdot V_d$.  Hence:
\begin{align*}
    4x + 2b + 2c + 2p & = \sum_{d \geq 3} d \cdot V_d \\
    & \geq 3V_3 + 4\sum_{d \geq 4} V_d \\
    & = 4\sum_{d \geq 3} V_d - V_3 \\
    & = 4x + 4b + 4c + 4p - V_3.
\end{align*}
Therefore, $V_3 \geq 2b + 2c + 2p$ where equality holds if and only if
$V_d = 0$ for $d \geq 5$.  If $p = 0$ then either $V_3 > 0$, and we
obtain the first conclusion, or $V_3 = 0$, and we have the second.  If
$p > 0$ then $V_3 \geq 2p$ and we obtain the first conclusion. 
\end{proof}

Let $\{\delta_i\}_{i = 1}^M$ enumerate the arcs of $\omega \cap (S -
A_\ell)$, where the order of the indices agrees with the orientation
of $\omega$.  So the system $A_{\ell+1}$ is obtained from $A_\ell$ via
surgery along $\delta_1$.  Generically, our strategy is to find a disk
component $R \subset S - A_\ell$ and an arc $\delta_i \subset R$ so
that
\begin{itemize}
\item
$\delta_i$ meets both $B_\ell$ and $C_\ell$ and
\item
$\delta_i$ is parallel in $R$ to a subarc of $\bdy S$.
\end{itemize}
That is, $\delta_i$ cuts a rectangle off of $R$.  Surgery along
$\delta_i$ then produces a common descendent for the systems $B$ and
$C$.

Suppose conclusion (1) of the claim holds.  Deduce there is a disk
component $R \subset S - A_\ell$ that is combinatorially a hexagon,
with sides alternating between $\partial S$ and $A_\ell$.
Furthermore, $R$ meets both $B_\ell$ and $C_\ell$.  As a very special
case, if $\delta_1$ lies in $R$ then take $k = \ell + 1$ and we are
done.  See the left-hand side of \reffig{Multicolored}.

\begin{figure}[htbp]
\[
\begin{array}{ccc}
\includegraphics[height = 3.5 cm]{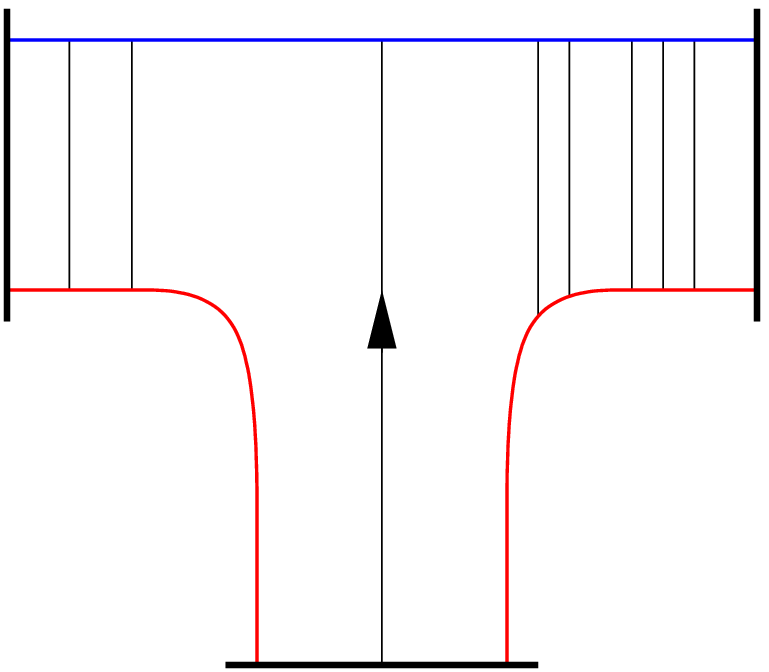} && 
\includegraphics[height = 3.5 cm]{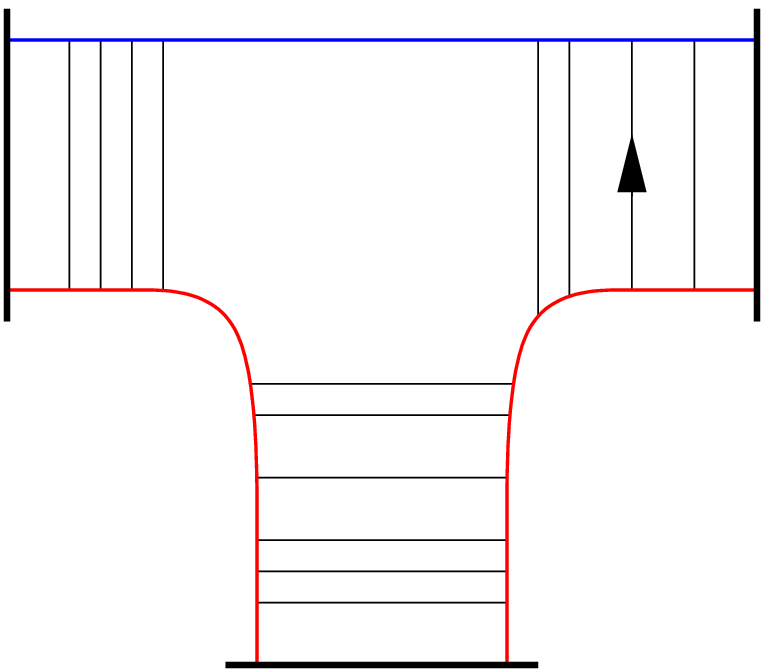}
\end{array}
\]
\caption{The lower and the vertical sides of $R$ lie in $\bdy S$; the
  longer boundary arcs lie in $A_\ell$.  The arcs in the interior are
  subarcs of $\omega$.  The arc with the arrow is $\delta_1$ on the
  left and is $\delta_m$ on the right.}
\label{Fig:Multicolored}
\end{figure}

If $\delta_1$ does not lie in $R$, then let $\delta_m$ be the first
arc contained in $R$ that meets both $B_\ell$ and $C_\ell$.  Set $k =
\ell + m$.  See the right-hand side of \reffig{Multicolored}.  One of
the arcs in $\frontier(R)$ survives to $A_{k - 2}$.  Thus $\indist(A,
A_i) \leq 3$ for all $i \in [\ell,k - 2]$.  The frontier of $R$ may be
surgered during the interval $[\ell+1,k-2]$, but there is always a
hexagon bounded by the children of $\frontier(R)$, containing the arc
$\delta_m$.  Surgering $\delta_m$ produces the desired common
descendants in $A_k$.  Finally, we note that $\indist(A, A_{k-1})$ and
$\indist(A, A_k)$ are at most $4$ as a child of an arc of
$\frontier(R)$ is in both $A_{k-1}$ and $A_k$.  Hence the lemma holds
in this case.

Suppose instead that conclusion (2) of the claim holds.  Thus every
component of $S - A'$ is combinatorially an octagon with sides
alternating between $\partial S$ and $A'$.  If $A_\ell \neq A'$ then
$S - A_\ell$ has a disk component that is combinatorially a hexagon,
and the above argument applies.  Therefore, we assume $A', B', C' =
A_\ell, B_\ell, C_\ell$.

Fix a component $R \subset S - A_\ell$ that does not contain
$\delta_1$.  
%%% Each octagon has Euler characteristic -1, so there are at least
%%% two of octagons. 
We refer to the four sides of $\frontier(R) \subset A_\ell$ using the
cardinal directions $\sN$, $\sS$, $\sE$ and $\sW$.  Up to
interchanging $B_\ell$ and $C_\ell$, there are three cases to
consider, depending on how $\sN$, $\sS$, $\sE$ and $\sW$ lie in
$B_\ell$ or $C_\ell$.

Suppose that $\sN$ lies in $B_\ell$ and the three other sides lie in
$C_\ell$.  Suppose there is an arc $\delta_i$ in $R$ connecting $\sN$
to $\sE$ or $\sN$ to $\sW$.  Let $\delta_m$ be the first such arc.
Arguing as before, under conclusion (1), the lemma holds.  If there is
no such arc then, as $\omega$ cuts $\sN$, there is an arc $\delta_i$
connecting $\sN$ to $\sS$.  Let $\delta_m$ be the first such arc; set
$k = \ell + m$.  As $\sN \in A_j$ for all $j \in [\ell, k-2]$, deduce
$\indist(A,A_i) \leq 3$ for all such $j$.  Also $\indist(A,A_{k-1})$
and $\indist(A,A_k)$ are at most $4$ as a child of an arc of
$\frontier(R)$ is in both $A_{k-1}$ and $A_k$.  We now observe that
some descendants of $\frontier(R)$ cobound a combinatorial hexagon
$R'$ in $S - A_k$.  If $\omega$ misses any arc in the frontier of
$R'$, then $\indist(A,\omega) \leq 5$, contrary to the hypothesis.
Else, arguing as in conclusion (1), the lemma holds.

Suppose $\sN$ and $\sE$ lie in $B_\ell$ while $\sS$ and $\sW$ lie in
$C_\ell$.  If there is an arc connecting $\sN$ to $\sW$ or connecting
$\sE$ to $\sS$, then surgery along the first such produces common
descendants.  If there is no such arc, then there must be an arc
connecting $\sN$ to $\sS$ or an arc connecting $\sE$ to $\sW$; if not
$\omega$ misses one of the diagonals of $R$, so $\indist(\omega,
A_\ell) \leq 2$ implying $\indist(\omega, A) \leq 5$, contrary to
assumption.  Again, surgery along the first such arc produces a
combinatorial hexagon.

Suppose finally that $\sN$ and $\sS$ lie in $B_\ell$ while $\sE$ and
$\sW$ lie in $C_\ell$.  Surgery along the first arc connecting
$B_\ell$ to $C_\ell$, inside of $R$, produces common descendants.
Such an arc exists because $\omega$ cuts every arc of $A_\ell$. 
\end{proof}

\section{Footprints}
\label{Sec:Footprints}

In this section we define the \emph{footprint} of an arc or curve on a
surgery sequence.  This is not to be confused with the projection,
which is defined in \refsec{AC}.

Fix $\gamma \in \AC(S)$.  Suppose $A$ is a system and $\omega$ is a
directed arc.  Let $\{ A_i \}_{i = 0}^N$ be the surgery sequence
starting at $A$ with target $\omega$.  We define $\phi(\gamma)$, the
\emph{footprint} of $\gamma$ on $\{A_i\}$, to be the set
\[
\phi(\gamma) = \{ i \in [0,N] \st \mbox{$\gamma$ misses $A_i$}\}.
\]
Note that if $\gamma$ is an element of $A_i$ then $i$ lies in the
footprint $\phi(\gamma)$.

\begin{lemma}
\label{Lem:Foot}
With $\gamma, A, \omega$ as above: the footprint $\phi(\gamma)$ is an
interval. 
\end{lemma}

\begin{proof}
When $\gamma$ is a curve, this follows from \reflem{Cut}.  So suppose
that $\gamma$ is an arc.  Without loss of generality we may assume
$\phi(\gamma)$ is non-empty and $\min \phi(\gamma) = 0$.  Note that if
$\omega$ misses $\gamma$ then we are done.  Isotope $\gamma$, $A$, and
$\omega$ to minimize their intersection numbers.

We now surger $A_0 = A$.  These surgeries are ordered along $\omega$.
Let $\alpha_i$ be the active element of $A_i$.  Let $\delta_i \subset
\omega$ be the surgery arc for $\alpha_i$, in other words, the subarc
of $\omega$ with endpoints the initial endpoint of $\omega$ and the
initial intersection point between $\omega$ and $\alpha_i$.  We
define a pair of intervals. 
\begin{align*}
I &= \{i \st \delta_{i-1} \cap \gamma = \emptyset\} \cup \{0\} \\
J &= \{i \st \delta_{i-1} \cap \gamma \neq \emptyset\}
\end{align*}
The inclusions $\delta_{i-1} \subset \delta_i$ and the fact that
$\gamma$ misses $A_0$ implies that $I \subset \phi(\gamma)$.  To
finish the proof we will show $J \cap \phi(\gamma) = \emptyset$,
implying that $I = \phi(\gamma)$.  

Fix any $k \in J$.  Let $\alpha_{k-1}$ be the active element of
$A_{k-1}$.  As $\alpha_{k-1}$ is an arc or a curve we consult the
left- or right-hand side of \reffig{Surgery}.  Note that $\gamma$
meets $\delta_{k-1}$, and $\gamma$ is an arc, so it enters and exits
the region cobounded by $\alpha_{k-1}$ and its children.  Thus
$\gamma$ cuts $A_k$ and we are done.
\end{proof}

\iffalse Since each hexagon has frontier in both $B'$ and $C'$, there
are arcs of $A_\ell$ that survive to $A_k$.

We call an arc $\delta_i$ \emph{multicolored} if
\begin{itemize}
\item
$\delta_i$ is contained in a hexagon of $S - A'$ and
\item
$\delta_i$ meets both $B'$ and $C'$.  
\end{itemize}
Since $\omega$ cuts every arc of $A_\ell$ it follows that every
hexagonal component of $S - A'$ contains a multicolored arc.  Let $m$
be the smallest index of the multicolored arcs.  Let $R$ be the
hexagonal region that contains $\delta_m$ and set $k = \ell + m$.  For
example, see the right-hand side of \reffig{Multicolored}.  Since each
hexagon has frontier in both $B'$ and $C'$, there are arcs of $A_\ell$
that survive to $A_k$.  Thus $\indist(A, A_i) \leq 3$ for all $i \in
[\ell,k]$.  The frontier of $R$ may be surgered during the interval
$[\ell+1,k-1]$, but there is always a hexagon bounded by the children
of $\frontier(R)$, containing the arc $\delta_m$.  This produces the
desired common descendants in $A_k$.
\fi

% Local Variables:
% mode: latex
% TeX-master: "arcs"
% End:

\section{Projections to surgery sequences}
\label{Sec:AC}

In Propositions~\ref{Prop:Retraction}, \ref{Prop:Lipschitz},
and~\ref{Prop:Contraction} below we verify that a surgery path has a
projection map satisfying three properties, called here the
\emph{retraction axiom}, the \emph{Lipschitz axiom}, and the
\emph{contraction axiom}.  These were first set out by Masur and
Minsky~\cite[Definition~2.2]{MasurMinsky99}.  We closely follow Handel
and Mosher~\cite{HandelMosher11}.  We also refer to the paper of
Hilion and Horbez~\cite{HilionHorbez12}.  We emphasize that the
various constants appearing in our argument are \emph{uniform}, that
is, independent of the surface $S = S_{g,n}$, mainly by virtue of
\reflem{Common}.

The relevance of the three axioms is given by the following theorem of
Masur and Minsky~\cite[Theorem~2.3]{MasurMinsky99}.

%%% Restate here and copy down. 
\begin{theorem}
\label{Thm:MM} 
If $\calX$ has an almost transitive family of paths, with projections
satisfying the three axioms, then $\calX^{(1)}$ is hyperbolic.
Furthermore, the paths in the family are uniform reparametrized
quasi-geodesics.
\end{theorem}

Before turning to definitions, we remark that the hyperbolicity
constant and the quasi-geodesic constants depend only on the constants
coming from almost transitivity and from the three axioms.  In
\refsec{Hyperbolicity} we provide a proof of \refthm{MM}, giving an
estimate for the resulting hyperbolicity constant.

\subsection{Transitivity}

Suppose that $\calX$ is a flag simplicial complex.  A \emph{path} is a
sequence $\{ \sigma_i \}_{i = 0}^N$ of simplices in $\calX$.  A family
of paths in $\calX$ is \emph{$d$--transitive} (or simply \emph{almost
  transitive}) if for any vertices $x, y \in \calX^{(0)}$ there exists
a path $\{ \sigma_i \}_{i=0}^N$ in the family such that $\indist(x,
\sigma_0)$, $\indist(\sigma_i,\sigma_{i+1})$, and $\indist(\sigma_N,
y)$ are all at most $d$.

\begin{lemma}[Transitivity]
\label{Lem:Transitive}
Surgery sequences form a $2$--transitive family of paths. 
\end{lemma}

\begin{proof}
  Fix $\alpha, \beta \in \AC(S)$.  Pick an oriented arc $\omega \in
  \AC(S)$ so that $\I(\beta, \omega) = 0$.  Let $\{A_i\}_{i=0}^N$ be
  the surgery sequence starting at $A_0 = \{\alpha\}$ with target
  $\omega$.  Since $\indist(A_N,\beta) \leq 2$, the lemma is proved.
\end{proof}

%%% 2 is sharp.  Pf: Suppose that \beta cuts P, a pants, off of S.
%%% Suppose that \bdy P = \beta \cup \bdy S.  That is, |\bdy S| = 2.
%%% Suppose that \alpha \cap P is a parallel collection of waves.  (So
%%% \alpha is a curve, not an arc.)  Now, \omega must be either a seam
%%% or a wave in P (there are two possible waves, but there is a
%%% mapping class interchanging them).  For either choice of \omega,
%%% all elements of all A_i will intersect \beta.

\subsection{Projection}

We now define the projection map to a surgery sequence, following
Handel and Mosher, see \reffig{Projection}.  We then state and verify
the three axioms in our setting.

\begin{definition}[Projection]
\label{Def:Projection}
Suppose $\{ A_i \}_{i = 0}^N$ is a surgery sequence with target
$\omega$.  We define the \emph{projection map} $\pi \from \AC(S) \to
[0,N]$ as follows.  Fix $\beta \in \AC(S)$.  Suppose that $\{B_j\}$ is
the surgery sequence starting at $B = \{\beta\}$ with target $\omega$.
Define $\pi(\beta)$ to be the least index $m \in [0,N]$ so that there
is an index $k$ with $A_m \cap B_k \neq \emptyset$.  If no such index
$m$ exists then we set $\pi(\beta) = N$.
\end{definition}

%%% Should that last be N+1?

In the following we use the notation $[i,j] =
[\min\{i,j\},\max\{i,j\}]$ when the order is not important.  We
also write $A[i,j]$ for the union $\cup_{k \in [i,j]} A_k$.

\begin{proposition}[Retraction]
\label{Prop:Retraction}
For any surgery sequence $\{A_i\}_{i = 0}^N$, index $k \in [0,N]$, and
element $\beta \in A_k$ we have the diameter of $A[\pi(\beta), k]$ is
at most two.
\end{proposition}

%%% This is sharp.

\begin{proof}
Let $\{ B_j \}_{j = k}^N$ be the surgery sequence subordinate to
$\{A_i\}_{i = k}^N$ that starts at $B = \{\beta\}$.  Set $m =
\pi(\beta)$; note that $m \leq k$, as $\beta \in B_k \subset A_k$.

Suppose that $A_m \cap B_\ell \neq \emptyset$ for some $\ell \geq k$.
As $\{ B_j \}$ is subordinate to $\{ A_i \}$ we have $B_\ell \subset
A_\ell$.  Pick any $\gamma \in A_m \cap A_\ell$.  By \reflem{Foot} we
have that $[m,\ell]$ lies in $\phi(\gamma)$, the footprint of
$\gamma$.  Thus $[m,k]$ lies in $\phi(\gamma)$.  Thus the diameter of
$A[m,k]$ is at most two, finishing the proof.
\end{proof}

Instead of using footprints, Hilion and
Horbez~\cite[Proposition~5.1]{HilionHorbez12} verify the retraction
axiom by using the fact that intersection numbers decrease
monotonically along a surgery sequence.

The verification of the final two axioms is identical to that of
Handel and Mosher~\cite{HandelMosher11}: replace their Proposition~6.5
in the argument of Section~6.3 with \reflem{Common}.  Alternatively,
in the geometric setting these arguments appear in Section~7
of~\cite{HilionHorbez12}: replace their Proposition~7.1 with our
\reflem{Common}.
  
\begin{proposition}[Lipschitz]
\label{Prop:Lipschitz}
For any surgery sequence $\{A_i\}_{i = 0}^N$ and any vertices
$\beta,\gamma \in \AC(S)$, if $d_S(\beta,\gamma) \leq 1$ then the
diameter of $A[\pi(\beta), \pi(\gamma)]$ is at most $14$.
\end{proposition}

\begin{proof}
  Let $m = \pi(\beta)$ and $k = \pi(\gamma)$.  Without loss of
  generality we may assume that $m \leq k$.  There are two cases.
  Suppose that $\indist(A_m, \omega ) \leq 6$.  By
  \refprop{Monotonic}, for all $i \geq m$ we have $\indist(A_i,
  \omega) \leq 6$.  It follows that the diameter of $A[m,k]$ is at
  most $14$.

  Suppose instead that $\indist(A_m, \omega ) \geq 7$.  Fix some
  $\beta' \in A_m$, a descendent of $\beta$.  Thus there is a
  descendent $\gamma'$ of $\gamma$ with $d_S(\beta', \gamma') \leq 1$.
  Set $B' = \{\beta', \gamma'\}$ and note that $\indist(B', \omega)
  \geq 6$.  Let $\{B_i'\}$ be the resulting surgery sequence with
  target $\omega$.

  By \reflem{Common}, there is an index $p$ and some $\delta \in B_p'$
  that is a common descendent of both $\beta'$ and $\gamma'$.
  Additionally, any vertex of $B'[0,p]$ has inner distance to $B' =
  B_0'$ of at most five.  Now, since $\delta$ is a descendent of
  $\beta'$ there is some least index $q$ so that $\delta \in A_q$.
  Thus $k \leq q$.  It follows that the diameter of $A[m,k]$ is at
  most $14$.
\end{proof} 

\begin{proposition}[Contraction]
  \label{Prop:Contraction}
  There are constants $a, b, c$ with the following property.  For any
  surgery sequence $\{A_i\}_{i = 0}^N$ and any vertices $\beta, \gamma
  \subset \AC(S)$ if
  \begin{itemize}
  \item $\indist(\beta, A[0,N]) \geq a$ and
  \item $d_S(\beta, \gamma) \leq b \cdot \indist(\beta, A[0,N])$
  \end{itemize}
  then the diameter of $A[\pi(\beta),\pi(\gamma)]$ is at most $c$.
  
  In fact, the following values suffice: $a = 24$, $b = \frac{1}{8}$
  and $c = 14$.
\end{proposition}

\begin{proof}
Suppose $\{ A_i \}_{i=0}^N$ is a surgery sequence with target
$\omega$.  Let $\pi \from \AC(S) \to [0,N]$ denote the projection to
the surgery sequence $\{ A_i \}$.  Let $\{ B_j \}_{j = 0}^M$ be the
surgery sequence starting with $B_0 = \{\beta \}$ with target
$\omega$.

The contraction axiom is verified by repeatedly applying
\reflem{Common}: if two arcs or curves are far from $\{A_i\}_{i=0}^N$
but proportionally close to one another, then their surgery sequences
have a common descendant prior to intersecting $\{ A_i \}$. An
application of the Lipschitz axiom, \refprop{Lipschitz}, completes the
proof.

We begin with a claim.  For the purpose of the claim, we use weaker
hypotheses: $\indist(\beta, A[0,N]) \geq 21$ and $d_S(\beta,\gamma)
\leq \frac{1}{7}\indist(\beta,A[0,N])$.
\begin{claim*}
There is an index $k \in [0,M]$ so that
\begin{itemize}
\item 
$B_k$ contains a descendent of $\gamma$ and
\item $\indist(\beta, B_j) \leq 6d_S(\beta,\gamma)$ for all $j \in
  [0,k]$.
\end{itemize}
\end{claim*}
  
\begin{proof}[Proof of Claim]
%%% Base case???
  Fix $\alpha \in \AC(S)$ such that $d_S(\beta,\alpha) = d_S(\beta,
  \gamma) - 1$ and $\I(\alpha,\gamma) = 0$.  By induction, there is an
  index $\ell \in [0,M]$ such that $B_\ell$ contains a descendent of
  $\alpha$ and such that $\indist(\beta, B_j) \leq 6d_S(\beta, \alpha)
  = 6d_S(\beta ,\gamma) - 6$ for all $j \in [0,\ell]$.  Let $\beta'
  \in B_\ell$ be such a descendent.  As $\I(\alpha,\gamma) = 0$, it
  follows that $\gamma$ has a descendant, $\gamma'$, that misses
  $\beta'$.  Let $B' = \{\beta', \gamma'\}$ and let $\{B'_i\}$ be the
  resulting surgery sequence with target $\omega$.

We have:
\begin{align*}
  \indist(B', \omega ) &\geq \indist(B_\ell,\omega) - 1 \\
  &\geq d_S( \beta , \omega ) - \indist(\beta, B_\ell) - 2 \\
  &\geq \indist( \beta , A_N) - 1 - \left(6d_S(\beta, \gamma) - 6 \right) - 2 \\
  &\geq \frac{1}{7} \indist( \beta , A_N ) + 3 \\
  &\geq 6.
\end{align*}

As in the proof of \refprop{Lipschitz}, we use \reflem{Common} to
obtain an index $p$ and element $\delta \in B'_p$, so that $\delta$ is
a common descendent of $\beta'$ and $\gamma'$.  Additionally, any
element of $B'[0,p]$ has inner distance to $B'$ of at most five.  Let
$k \in [\ell,M]$ be the first index such that $\delta \in B_k$.

What is left to show is that for $j \in [\ell,k]$ we have
$\indist(\beta, B_j) \leq 6d_S(\beta, \gamma)$; by induction it holds
for $j \in [0,\ell]$.  As for each $j \in [\ell,k]$ the system $B_j$
contains a descendent of $\beta'$ we have:
\begin{align*}
  \indist(\beta,B_j) & \leq \indist(\beta, B') + \indist(B',B_j) + 1 \\
  & \leq (6d_S(\beta,\gamma) - 6) + 5 + 1 \\
  & \leq 6d_S(\beta, \gamma).
\end{align*}
This completes the proof of the claim.
\end{proof}
    
We now complete the verification of the contraction axiom.  There are
two cases.  Suppose $\pi(\beta) \leq \pi(\gamma)$ and the weaker
hypotheses hold: $\indist(\beta, A[0,N]) \geq 21$ and
$d_S(\beta,\gamma) \leq \frac{1}{7}\indist(\beta,A[0,N])$.  Let $k \in
[0,M]$ be as in the claim and let $\gamma_1 \in B_k$ be a descendent
of $\gamma$.  As $\gamma_1$ is a descendant of $\gamma$, we have that
$\pi(\gamma) \leq \pi(\gamma_1)$.  Let $\ell \in [0,N]$ be such that
$\indist( \beta , A_\ell)$ is minimal.  For all $j \in [0,k]$, by the
second bullet of the claim we have:
\begin{align*}
\indist(\beta, B_j ) & \leq 6d_S(\beta,\gamma) \\
  & \leq \frac{6}{7}\indist( \beta , A_\ell) \\
  & \leq \indist(\beta, A_\ell) - 2.
\end{align*}
Therefore, we have that $B_j \cap A_i = \emptyset$ for all $j \in
[0,k]$ and $i \in [0,N]$ and so $\beta$ has a descendant $\beta_1 \in
B_k$ such that $\pi(\beta) = \pi(\beta_1)$. Hence $[\pi(\beta),
\pi(\gamma)] \subset [\pi(\beta_1), \pi(\gamma_1)]$. By
\refprop{Lipschitz} as $d_S(\beta_1,\gamma_1) \leq 1$, the diameter of
$A[\pi(\beta_1), \pi(\gamma_1)]$ is at most 14. Therefore the diameter
of $A[\pi(\beta), \pi(\gamma)]$ is also at most 14.

We now deal with the remaining case.  Suppose $\pi(\beta) >
\pi(\gamma)$, $\indist(\beta, A[0,N]) \geq 24$ and $d_S(\beta,\gamma)
\leq \frac{1}{8}\indist(\beta,A[0,N])$.  Here we proceed along the
lines of \cite[Lemma~3.2]{HandelMosher11}.  We find for all $i \in
[0,N]$:
\begin{align}
  \indist(\gamma, A_i) &\geq \indist(\beta, A_i) - d_S(\beta,\gamma) \nonumber \\
  &\geq \frac{7}{8}\indist(\beta,A_i) \geq 21 \\
  \noalign{\noindent and} d_S(\beta,\gamma) & \leq \frac{1}{8}
  \indist(\beta, A_i) \leq \frac{1}{7} \indist(\gamma, A_i)
\end{align}
As $\pi(\gamma) \leq \pi(\beta)$, the above argument now implies that
the diameter of $A[\pi(\beta), \pi(\gamma)]$ is at most 14.
\end{proof}

% Local Variables:
% mode: latex
% TeX-master: "arcs"
% End:

\section{Hyperbolicity}
\label{Sec:Hyperbolicity}

In this section, we use the contraction properties of $\AC^{(1)}(S)$
to prove it is Gromov hyperbolic.  This is already proven in
\cite{MasurMinsky99}.  However, we need an explicit estimate for the
hyperbolicity constant.  Hence, we reproduce the argument here,
keeping careful track of constants.

We say a path $g \from I \to \calX$ is $(\ell,L)$--\emph{Lipschitz} if
\[
\frac{|s-t|}{\ell} \leq d_\calX\big( g(s) , g(t) \big) \leq L |s-t|.
\]
Let $a$, $b$ and $c$ be the constants from \refprop{Contraction}.

\begin{proposition}
\label{Prop:Surgery-Geodesic}
Suppose $g \from [0,M] \to \AC^{(1)}(S)$ is $(\ell,L)$--Lipschitz and
let $\{A_i\}_{i=0}^N$ be a surgery sequence so that $g(0)$ misses
$A_0$ and $g(M)$ misses $A_N$. Then, for every $t \in [0,M]$,
\[
d_\AC \big(g(t), \{A_i\} \big) \leq  \frac{4c\ell L (\ell L + 1)}{b},
\]
assuming $\displaystyle \frac{2c\ell L}{b}\geq a$.
\end{proposition}

\begin{remark}
Note that the hypothesis $\frac{2c\ell L}{b}\geq a$ holds for the
constants $a$, $b$ and $c$ given by \refprop{Contraction} if $\ell,L
\geq 1$.
\end{remark}

\begin{proof}[Proof of Proposition~\ref{Prop:Surgery-Geodesic}]
For $t \in [0,M]$, let $g_t=g(t)$. Define
\begin{equation} \label{Eqn:D}
D=\frac{2c \ell L}b,
\end{equation}
and let $I \subset [0,M]$ be an interval so that
for $t \in I$, $d_\AC ( g_t, \{A_i\}) \geq D$. Divide $I$ to
intervals of size at most $bD/L$. Assume there are $m$ such intervals with
\[
\frac{(m-1)bD}L \leq |I| \leq \frac{mbD}L.
\]
Note that the image of every subinterval $J$ under $g$ has a length
of $bD/L \leq b D$ and the whole interval is distance at least $D\geq
a$ from the surgery path $\{A_i\}$. Hence, \refprop{Contraction}
applies; so $\pi(g(J))$ has a diameter of at most $c$.  Let $R$ be
the largest distance between a point in $g(I)$ to the set
$\{A_i\}$. Since $g(0)$ and $g(M)$ are within distance $D$ of the set
$\{A_i\}$, we have
\[
R \leq D+ \frac{L|I|}2.
\]
Also, since $g$ is a $(\ell,L)$--quasi-geodesic, the end points of
$g(I)$ are at least $|I|/\ell$ apart. That is, 
\[
\frac{(m-1)bD}{\ell L} \leq  \frac{|I|}{\ell} \leq mc + 2D. 
\]
Thus, 
\[
m (bD - c \ell L) \leq 2 \ell L D + bD \quad\Longrightarrow\quad
m \leq \frac{D(2\ell L+b)}{bD-c \ell L}.
\]
This, in turn, implies that
\[
R \leq D+  \frac{\ell L (mc+2D)}2 \leq  
(\ell L +1)D + \frac{c \ell L D(2\ell L+b)}{2(bD-c \ell L)}.
\]
From \refeqn{D} we get
\[
R \leq (\ell L +1)D + D(\ell L + b/2) \leq D (2 \ell L + 2) 
= \frac{4c\ell L (\ell L + 1)}{b},
\]
which is as we claimed. 
\end{proof}

\begin{theorem}
\label{Thm:AC}
If $3g - 3 + n \geq 2$ and $n > 0$, then $\AC^{(1)}(S_{g,n})$ is
$\delta$--hyperbolic where
\[
\delta = \frac{56c}b + \frac c2 + 1.
\]
\end{theorem}

\begin{proof}
Consider three points $\alpha, \beta, \gamma \in \AC^{(1)}(S_{g,n})$.
Choose a geodesic segment connecting $\alpha$ to $\beta$ and denote it
by $[\alpha,\beta]$. Let $[\beta, \gamma]$ and $[\alpha,\gamma]$ be
similarly defined. We need to show that the geodesic segment $[\beta,
  \gamma]$ is contained in a $\delta$--neighborhood of $[\alpha,
  \beta] \cup [\alpha, \gamma]$.

Let $\alpha'$ be the closest point in $[\beta, \gamma]$ to $\alpha$.
The path $p_{\alpha,\beta}$ obtained from the concatenation $[\alpha,
  \alpha'] \cup [\alpha', \beta]$ is a $(3,1)$--Lipschitz
path~\cite[page~147]{MasurMinsky99}.
%%% Check this
By \refprop{Surgery-Geodesic}, 
\[
\text{if} \quad \ell=3, L=1, \quad \text{then} \quad R\leq \frac{48c}{b}.
\]
That is, $p_{\alpha,\beta}$ stays in a $(48c/b)$--neighborhood of any surgery 
path $\{A_i\}$ that starts next to $\alpha$ and end next to $\beta$.
(Recall that surgery paths are $2$--transitive.) Also by
\refprop{Surgery-Geodesic},
\[
\text{if} \quad \ell=L=1, \text{then} \quad R \leq \frac{8c}{b}.
\]
That is, the geodesic $[\alpha,\beta]$, which is a $(1,1)$--Lipschitz
path, stays in a $(8c/b)$--neighborhood of $\{A_i\}$. By the Lipschitz
property of projection, its image is $c$ dense. That is, any point in
$\{A_i\}$ is at most $8c/b+\frac c2 +1$ away from a point in
$[\alpha,\beta]$. Therefore, the path $p_{\alpha,\beta}$ is contained
in a
\[
\delta = \frac{48c}{b}+\frac{8c}{b}+\frac{c}{2}+1
\]
neighborhood of $[\alpha, \beta]$. Similar arguments shows that the
path $p_{\alpha, \gamma}$ is contained in a $\delta$--neighborhood of
$[\alpha, \gamma]$.  Hence, $[\beta,\gamma]$ is contained in a
$\delta$--neighborhood of $[\alpha, \beta] \cup [\alpha, \gamma]$.
That is, $\AC^{(1)}(S)$ is $\delta$--hyperbolic.
\end{proof}

% Local Variables:
% mode: latex
% TeX-master: "arcs"
% End:

\section{Inclusions}
\label{Sec:Include}

In this section, we show that the hyperbolicity of the curve complex follows from
the hyperbolicity of the arc and curve complex. 

\begin{corollary}
\label{Cor:Curve}
There is a constant $\Uniform$ such that if $3g - 3 + n \geq 2$ and $n >
0$ then $\calC^{(1)}(S_{g,n})$ is $\Uniform$--hyperbolic.
\end{corollary}

\begin{proof}
The surgery relation $\sigma \from \AC \to \calC$ takes curves to
themselves and sends an arc $\alpha$ to a system $A = \sigma(\alpha)$
so that $\alpha$ is contained in a pants component of $S - A$.  For
$\alpha, \beta \in \AC$ we have
\[
d_\calC(\sigma(\alpha),\sigma(\beta)) \leq 2 d_\AC(\alpha,\beta)
\]
by Lemma~2.2 of~\cite{MasurMinsky00}.  On the other hand, for $\alpha,
\beta \in \calC$ we have
\[
d_\AC(\alpha,\beta) \leq d_\calC(\alpha,\beta).
\]
Thus the inclusion of $\calC^{(1)}(S_{g,n})$ into $\AC^{(1)}(S_{g,n})$
sends geodesics to $(1,2)$--Lipschitz paths. Continuing as 
in the proof of \refthm{AC}, we get that the image of a geodesic in 
$\calC$ is in a uniformly bounded neighborhood of a geodesic in $\AC$. 
Hence, the hyperbolicity of $\AC$ implies the hyperbolicity of $\calC$.
\end{proof}

We now deal with the case when $S = S_g$ is closed.

\begin{theorem}
\label{Thm:Closed}
If $3g - 3 \geq 2$ then $\calC^{(1)}(S_{g})$ is Gromov hyperbolic.
Furthermore, the constant of hyperbolicity is at most that of
$\calC^{(1)}(S_{g,1})$.
\end{theorem}

\begin{proof}
Let $\Sigma = S_{g,1}$.  By \refcor{Curve} we have
$\calC^{(1)}(\Sigma)$ is $\Uniform$--hyperbolic.  By Theorem~1.2
of~\cite{RafiSchleimer09}, the curve complex $\calC^{(1)}(S)$
isometrically embeds in the curve complex $\calC^{(1)}(\Sigma)$.
%%% Suppose PQR is a geodesic triangle in \calC(S).  Then it is again
%%% a geodesic triangle in \calC(\Sigma).  Pick any x in the side R.
%%% Then there is a point y, say in the side P, so that d_\Sigma(x,y)
%%% \leq \Uniform.  It follows that d_S(x,y) \leq \Uniform.
\end{proof}

\bibliographystyle{hyperplain} % replaces plain.bst
\bibliography{bibfile}
\end{document}